\documentclass[11pt]{amsart}
\usepackage{float}
\usepackage[usenames]{color}
\usepackage{graphicx}
\usepackage{amssymb}
\usepackage{amscd}
\usepackage{makecell}
\usepackage{caption}
\usepackage{url}
\usepackage{mathtools}
\usepackage{mathrsfs}
\makeatletter
\newcommand*\bigcdot{\mathpalette\bigcdot@{.5}}
\newcommand*\bigcdot@[2]{\mathbin{\vcenter{\hbox{\scalebox{#2}{$\m@th#1\bullet$}}}}}
\makeatother%%https://tex.stackexchange.com/questions/235118/making-a-thicker-cdot-for-dot-product-that-is-thinner-than-bullet

\usepackage{tikz, tikz-3dplot, pgfplots}
\pgfplotsset{compat=1.8}
\usepackage{dynkin-diagrams}%%https://ctan.org/pkg/dynkin-diagrams

\theoremstyle{plain}
\newtheorem{thm}{Theorem}[section]
\newtheorem{lemma}[thm]{Lemma}
\newtheorem{cor}[thm]{Corollary}

\theoremstyle{definition}
\newtheorem{defi}[thm]{Definition}

\newtheorem{example}[thm]{Example}

\def\dim{\mathop{\hbox {dim}}\nolimits}

\def\ch{\mathop{\hbox {ch}}\nolimits}

\def\im{\mathop{\hbox {Im}}\nolimits}

\def\ker{\mathop{\hbox{Ker}}\nolimits}

\newcommand{\frg}{\mathfrak{g}}
\newcommand{\frh}{\mathfrak{h}}

\newcommand{\frk}{\mathfrak{k}}

\newcommand{\fro}{\mathfrak{o}}
\newcommand{\frp}{\mathfrak{p}}

\newcommand{\frs}{\mathfrak{s}}
\newcommand{\frt}{\mathfrak{t}}

\newcommand{\fk}{\mathfrak{k}}

\newcommand{\fp}{\mathfrak{p}}

\newcommand{\bbC}{\mathbb{C}}

\newcommand{\bbN}{\mathbb{N}}

\newcommand{\bbR}{\mathbb{R}}

\newcommand{\bbZ}{\mathbb{Z}}

\newcommand{\caC}{\mathcal{C}}

\newcommand{\caS}{\mathcal{S}}

\newcommand{\caU}{\mathcal{U}}

\newcommand{\be}{\begin {equation}}
\newcommand{\ee}{\end {equation}}
\newcommand{\bp}{\begin {proof}}
\newcommand{\ep}{\end {proof}}

\begin{document}

\title[Dirac cohomology, branching laws and Wallach modules]
{Dirac cohomology, branching laws and Wallach modules}

\author{Chao-Ping Dong}
\address[Dong]{School of Mathematical Sciences, Soochow University, Suzhou 215006,
P.~R.~China}

\email{chaopindong@163.com}

\author{Yongzhi Luan}
\address[Luan]{School of Mathematics, Shandong University, Jinan 250100, P.~R.~China}
\email{luanyongzhi@email.sdu.edu.cn}

\author{Haojun Xu}
\address[Xu]{School of Mathematical Sciences, Soochow University, Suzhou 215006,
P.~R.~China}
\email{20234207003@stu.suda.edu.cn}

\abstract{The idea of using Dirac cohomology to study branching laws was initiated by Huang, Pand\v zi\'c and Zhu in 2013 \cite{HPZ}.  One of their results says that the Dirac cohomology of $\pi$ completely determines $\pi|_{K}$, where $\pi$ is any irreducible unitarizable highest weight $(\frg, K)$ module. This paper aims to develop this idea for the exceptional Lie groups $E_{6(-14)}$ and $E_{7(-25)}$: we recover the $K$-spectrum of the Wallach modules from their Dirac cohomology.}
 \endabstract



\maketitle
\section{Introduction}
Let $G$ be a connected simple real Lie group with a maximal compact subgroup $K:=G^{\theta}$, where $\theta$ is a Cartan involution of $G$. Although many results quoted below hold in a wider setting, for simplicity, we assume that the pair $(G, K)$ is Hermitian symmetric. Let $\frg_0$ and $\frk_0$ be the Lie algebra of $G$ and $K$, respectively. Then
$$
\frg_0=\frk_0 +\frp_0
$$
is the Cartan decomposition on the Lie algebra level, and
$$
G=K \, \rm{exp}(\frp_0)
$$
is the Cartan decomposition on the Lie group level.
Fix a maximal torus $T$ of $K$, and let $\frt_0$ be its Lie algebra. Note that $T$ is a maximally compact $\theta$-stable Cartan subgroup of $G$. We will drop the subscript $0$ from a real Lie algebra to denote its complexification. We fix a positive root system $\Delta_{\fk}^+ = \Delta^+(\frk, \frt)$ once for all. Choose a positive roots system $\Delta_{\frg}^+$ of $\Delta(\frg, \frt)$ containing $\Delta_{\fk}^+$. Let $\Delta_{n}^{+}:= \Delta_{\frg}^+\setminus \Delta_{\fk}^+$ and put $\Delta_{n}^{-}=-\Delta_{n}^{+}$. Let $\frp^+=\oplus_{\alpha\in \Delta_{n}^{+}}  \frg_{\alpha}$ and $\frp^-=\oplus_{\alpha\in \Delta_{n}^{-}}  \frg_{\alpha}$. Then $[\frp_+, \frp_+]=[\frp_{-}, \frp_{-}]=0$ and
$$
\frp=\frp_+ \oplus \frp_{-}.
$$
We denote the half sum of positive roots in $\Delta_{\frg}^+$, $\Delta_{\fk}^+$, and $\Delta_n^{+}$ by $\delta$, $\delta_\fk$, and $\delta_n$, respectively.

Let $E_\mu$ be a \emph{$K$-type}, that is, a finite-dimensional irreducible $K$-module with highest weight $\mu$.
The \emph{generalized Verma module} is the $(\frg, K)$-module defined as
\begin{equation}\label{gen-Verma}
N(\mu) \coloneqq \caU(\frg) \otimes _{\caU(\fk \oplus \fp^+)} E_\mu,
\end{equation}
%where $\caU(\frl)$ stands for the universal enveloping algebra of $\frl$ and $\fp^+$ act trivially on $E_\mu$.%%$\frl$ is usually used for the Levi part, the notion may casue confusion.
%where $\caU(\mathfrak{x})$ stands for the universal enveloping algebra of $\mathfrak{x}$, and $\fp^+$ act trivially on $E_\mu$.
where $\caU(\bigcdot)$ stands for the universal enveloping algebra, and $\fp^+$ acts trivially on $E_\mu$.
Thus as $\frk$-modules, we have
\begin{equation}\label{gen-Verma-k}
N(\mu) \cong S(\frp^-) \otimes  E_\mu.
\end{equation}
The $\frk$-module structure of the symmetric algebra $S(\frp^-)$ was determined by Schmid \cite{S}.

The unique irreducible quotient of $N(\mu)$ containing $E_{\mu}$ is denoted by $L(\mu)$, and is called an \emph{irreducible highest weight module}. Any irreducible $(\frg, K)$-module which is also a highest weight $\frg$-module is isomorphic to some $L(\mu)$, which has infinitesimal character $\mu+\delta$. The unitary ones among these $L(\mu)$ were classified in the 1980s by Enright, Howe, Wallach \cite{EHW}, and by Jakobsen \cite{Ja}. It turns out that, among all the irreducible unitary highest weight modules, the (a few non-trivial) \emph{Wallach modules} are basic ones. For instance, $E_{6(-14)}$ has one Wallach module, while $E_{7(-25)}$ has two. They are unipotent representations, and can not be obtained via cohomological induction.

As we shall state in Section \ref{sec-Dirac}, the Dirac operator for real reductive Lie groups was introduced by Parthasarathy in 1972 \cite{Pa}, and it was further sharpened to be Dirac cohomology by Vogan in 1997 \cite{Vog97}. After the verification of Vogan conjecture by Huang and Pand\v zi\'c in 2002 \cite{HP}, classifying the \emph{Dirac series} (that is, irreducible unitary representations with non-zero Dirac cohomology) became an interesting open problem. As pointed out in \cite{HPP}, Dirac series include all the irreducible unitary highest weight modules when $(G, K)$ is Hermitian symmetric.

Recently, the Dirac series of many groups have been classified. See \cite{BDW} for complex classical Lie groups, and see \cite{DW} for $GL(n, \bbR)$. On the other hand, it was initiated by Huang, Pand\v zi\'c and Zhu about a decade ago that Dirac cohomology can be utilized to study branching laws. One of their results, quoted as Theorem \ref{thm-HPZ} below, says that one can use the Dirac cohomology of $L(\mu)$ to recover the entire $K$-spectrum of $L(\mu)$. This result is quite surprising since among all the $K$-types of $L(\mu)$, only \emph{a few} (namely, the spin lowest $K$-types of $L(\mu)$ in the language of \cite{D}) contribute to $H_D(L(\mu))$. In other words, the (finitely many) spin LKTs of $L(\mu)$ completely determine the entire $K$-spectrum. This paper aims to carry out this idea for $E_{6(-14)}$ and $E_{7(-25)}$. Indeed, after developing some branching results, we shall recover the $K$-spectrum of the Wallach modules from their Dirac cohomology.

The paper is organized as follows: we collect necessary preliminaries on Dirac cohomology in Section \ref{sec-Dirac}. Then we develop some results on decomposing the tensor product of irreducible finite-dimensional modules in  Section \ref{sec-tensor-product}. We investigate $E_{6(-14)}$  in Section \ref{sec-E6} and  $E_{7(-25)}$ in Section \ref{sec-E7}.
Throughout this paper, we denote by $\bbN$ the set of non-negative integers.

\section{Preliminaries on Dirac cohomology}\label{sec-Dirac}

We keep the notations of the introduction.

\subsection{Dirac cohomology}

Let $B(\cdot,\cdot)$ be a nondegenerate invariant symmetric bilinear form on $\frg$. For the universal enveloping algebra $\caU (\frg)$ of $\frg$, and the Clifford algebra $\caC(\fp)$ of $\fp$ with respect to the bilinear form $B(\cdot, \cdot)$, the \emph{Dirac operator} introduced by Parthasarathy \cite{Pa} was defined as follows
\begin{equation}
    D \coloneqq \sum_{i=1}^{n} Z_i \otimes Z_i \, \in \caU(\frg) \otimes \caC(\frp) \,,
\end{equation}
where $\{ Z_1, \ldots , Z_n \}$ is an orthonormal basis of $\fp$ with respect to the bilinear form $B(\cdot, \cdot)$. Note that the Dirac operator $D$ is independent of the choice of the orthonormal basis $\{ Z_1, \ldots , Z_n \}$.

We denote the spinor module of the Clifford algebra $\caC(\fp)$ by $\caS$.
For any $(\frg, K)$-module $X$, the action of $\caU(\frg) \otimes \caC(\fp)$ on $X \otimes \caS$ is defined by
\[(u \otimes a) \cdot (x \otimes s) \coloneqq ux \otimes as \,, \]
where $u \in \caU(\frg)$, $a \in \caC(\fp)$, $x \in X$ and $s \in \caS$.
In fact, if we denote the spin double covering group of $K$ by $\tilde{K}$, then $X \otimes \caS$ is a $\left(\caU(\frg) \otimes \caC(\fp), \tilde{K} \right)$-module. See Lemma 3.2.2 of \cite{HP2} for details.

Since the Dirac operator $D$ is in $\caU(\frg) \otimes \caC(\fp)$, we know that $D$ acts on $X \otimes \caS$. Moreover, the action  $D: \, X \otimes \caS \longrightarrow X \otimes \caS$ is a $\tilde{K}$-module homomorphism.

Now \emph{Dirac cohomology} of $X$, as defined by Vogan \cite{Vog97}, is the following $\tilde{K}$-module
\begin{equation}\label{Dirac-coho}
H_D (X) \coloneqq \ker D / ( \im D \cap \ker D).
\end{equation}
If $H_D (X)$ is nonzero, the well-known Vogan conjecture says that the infinitesimal character of $X$ should be determined by $H_D (X)$. It was confirmed by Huang and Pand\v{z}i\'{c} in 2002 \cite{HP}.

\begin{thm}\emph{(\cite{HP})}\label{thm-HP}
Let $X$ be an irreducible $(\mathfrak{g}, K)$-module. If the Dirac cohomology $H_D(X)$ contains a $\tilde{K}$-type $E_{\mu}$ of highest weight $\mu \in \mathfrak{t}^{*}$, then the infinitesimal character of $X$ is $\mu + \delta_{\mathfrak{k}}$.
\end{thm}

It is interesting to note that when $X$ is unitary, the $\widetilde{K}$-modules $\im D$ and $\ker D$ intersect trivially and $H_D (X) = \ker D = \ker D^2$.

Because the dimension of the parabolic subalgebra $\fp$ is an even number, we can decompose the spinor module $\caS$ as the following $\fk$-modules
\[\caS = \caS^+ \oplus \caS^- \,,\]
where $\caS^+$ ($\caS^-$, respectively) is the direct sum of the even (odd, respectively) power of $\caS \simeq \bigwedge \fp^+$.
For any irreducible unitary $(\mathfrak{g},K)$-module $X$, we consider the following $\tilde{K}$-equivariant operator
\[ D^{\pm} : X \otimes \caS^{\pm} \longrightarrow X \otimes \caS^{\mp} \,,\]
which comes from the restriction of $D$
\[X \otimes \caS^{+} \xrightarrow{D} X \otimes \caS^- \xrightarrow{D} X \otimes \caS^{+} \,.\]
Since $\ker D \cap \im D = \{0\}$, we obtain
\begin{align*}
    X \otimes \caS & = \ker D \oplus \im D \,,\\
    X \otimes \caS^+ & = \ker D^+ \oplus \im D^- \,,\\
    X \otimes \caS^- & = \ker D^- \oplus \im D^+ \,,\\
\end{align*}
and an isomorphism
\[ D^{\pm} : \, \im D^{\mp} \longrightarrow \im D^{\pm} \,.\]
Let us define
\[ H_{D}^{\pm} = \ker D^{\pm} \,.\]
Then we have
\[H_D = H_D^{+} \oplus H_D^- \,.\]

\subsection{Dirac cohomology and $K$-character}
\begin{defi}\cite[page 1259]{HPZ}
Assume the $K$-type decomposition of the admissible $(\frg, K)$-module $X$ is
\[X = \bigoplus_{\lambda} \, m_\lambda E_{\lambda} \,.\]
The \emph{$K$-character} of $X$ is the following formal series
\begin{equation}
    \ch_K X = \sum_{\lambda} m_\lambda \ch_K E_\lambda \,,
\end{equation}
where $\ch_K E_\lambda$ is the character of the irreducible $K$-module $E_\lambda$.
\end{defi}
We will often work on the $\tilde{K}$-module in the following contents, but we still use $\ch_K$ to denote the corresponding $\tilde{K}$ character.

\begin{thm}\emph{\cite[Theorem A]{HPZ}}\label{thm-HPZ}
Let $N(\mu)$ be the generalized Verma module defined as in \eqref{gen-Verma}, and define the irreducible highest weight module $L(\mu)$ as in the introduction. Assume that  $L(\mu)$ is unitary. Let
\begin{equation}\label{eq:HPZthmAHD}
    H_D^+ (L(\mu)) = \sum_{\xi} E_\xi,\quad H_D^- (L(\mu)) = \sum_{\eta} E_\eta,
\end{equation}
where the summing indices $\xi$ and $\eta$ run over some finite subsets of the dominant $\fk$-weights. Then
\begin{equation}\label{eq:HPZthmA}
    \ch_K L(\mu) = \sum_{\xi} \ch_K N(\xi - \delta_n) - \sum_{\eta} \ch_K N(\eta - \delta_n).
\end{equation}
Here the summing indices $\xi$ and $\eta$ are the same as in \eqref{eq:HPZthmAHD}.
\end{thm}

Originally, Theorem \ref{thm-HPZ} was stated for lowest weight modules. Yet it holds for highest weight modules as well: one just needs to change $+\delta_n$ to be $-\delta_n$.

We emphasize that by Theorem \ref{thm-HP}, to compute $H_D(L(\mu))$, one only needs to look at the finitely many spin-lowest $K$-types of $L(\mu)$. Now Theorem \ref{thm-HPZ} says that we can get all the $K$-types of $L(\mu)$ from its spin-lowest $K$-types.

\section{Tensor product of finite-dimensional irreducible modules}\label{sec-tensor-product}

Let $\frg$ be a finite-dimensional simple Lie algebra over $\bbC$.
Fix a Cartan subalgebra $\frh$ of $\frg$. Then we have root system $\Delta(\frg, \frh)$. Let $W$ be the Weyl group of $\Delta(\frg, \frh)$.
We fix a positive root system $\Delta^+(\frg, \frh)$. Let $\Pi=\{\alpha_1, \alpha_2, \dots, \alpha_l\}$ be the corresponding simple roots, and let $\varpi_1, \varpi_2, \dots, \varpi_l$ be the corresponding fundamental weights. Let $\delta$ be half sum of the positive roots.
For integers $n_1, n_2, \dots, n_l$, we will use $[n_1, n_2, \dots, n_l]$ to stand for the vector $n_1\varpi_1+n_2\varpi_2+\cdots+n_l\varpi_l$. For instance, $\delta=[1,1, \dots, 1]$.

Let $E_{\lambda'}$ and $E_{\lambda^{\prime\prime}}$ be two irreducible finite-dimensional $\frg$-modules with highest weights $\lambda^\prime$ and $\lambda^{\prime\prime}$, respectively. This section aims to prepare results on the decomposition of $E_{\lambda^\prime}\otimes E_{\lambda^{\prime\prime}}$. The root systems that we adopt here are the same as those in Knapp \cite{Kn}. To make the following statements concise, we introduce the notation $\widetilde{E}_{\lambda}$
to stand for the $\frg$-module $E_{\lambda}$ if $\lambda$ is dominant integral with respect to $\Delta^+(\frg, \frh)$, and to stand for the zero module otherwise.

For an arbitrary vector $\mu\in\frh^*$, we denote by $\{\mu\}$ the unique element in $\frh^*$ which is dominant for $\Delta^+(\frg, \frh)$ and is conjugate to $\mu$ under the action of $W$.
We define $t(\mu)$ to be 0 if $\mu$ is singular for $\Delta(\frg, \frh)$; otherwise, there exists a unique $\sigma\in W$ such that $\sigma\mu=\{\mu\}$. In this case, we define $t(\mu)$ to be the sign of $\sigma$, that is, $(-1)^{l(\sigma)}$, where $l(\sigma)$ stands for the length of $\sigma$. The following result is Exercise 9 of \cite[Section 24]{Hum}.

\begin{lemma}\label{lemma-tensor-prod}
Let $\lambda'$ and $\lambda^{\prime\prime}$ be two dominant integral weights w.r.t. $\Delta^+(\frg, \frh)$.
We have that
\begin{equation}\label{tensor-prod-summands}
E_{\lambda'}\otimes E_{\lambda^{\prime\prime}}=\bigoplus_{\lambda\in \Pi(\lambda') }  m_{\lambda}\,   t(\lambda + \lambda^{\prime\prime}+\delta ) \, E_{\{\lambda+\lambda^{\prime\prime}+\delta\}-\delta}.
\end{equation}

\end{lemma}

\subsection{Type $D_n$}\label{sec-typeD}

Now we take $\frg$ to be $\frs\fro(2n, \bbC)$, which is of type $D_n$.

\begin{thm}\label{thm-branching-Dn}
Let $\frg$ be type $D_n$, and let $a_1, a_{n-1}\in\bbN$. Then we have that
\begin{itemize}
    \item [(a)] The tensor product $E_{[1, 0, \cdots, 0]}\otimes E_{[a_1, 0, \cdots,0, a_{n-1}, 0]}$ is multiplicity-free and consists of $\widetilde{E}_{[a_1, 0, \cdots,0, a_{n-1}, 0]+\nu}$, where $\nu$ runs over $\varpi_{1}$, $-\varpi_{1}$, $-\varpi_{n-1}+\varpi_n$, $-\varpi_1+\varpi_2-\varpi_{n-1}$.
     \item [(b)] The tensor product $E_{[0,  \cdots, 0, 1, 0]}\otimes E_{[a_1, 0, \cdots, 0, a_{n-1}, 0]}$ is multiplicity-free and consists of $\widetilde{E}_{[a_1, 0, \cdots,0, a_{n-1}, 0]+\nu}$, where $\nu$ runs over
     \begin{itemize}
         \item[$\bullet$]   $\varpi_{n-1}$, $-\varpi_{1}+\varpi_{n}$, $\varpi_{n-2i}-\varpi_{n-1}$ for $1\leq i < \frac{n}{2}$; $-\varpi_{1}+\varpi_{n-2i+1}-\varpi_{n-1}$ for $1< i \leq \frac{n}{2}$ if $n$ is odd;
         \item[$\bullet$]  $\varpi_{n-1}$, $-\varpi_{1}+\varpi_{n}$, $\varpi_{n-2i}-\varpi_{n-1}$ for $1\leq i <\frac{n}{2}$; $-\varpi_{1}+\varpi_{n-2i+1}-\varpi_{n-1}$ for $1< i \le \frac{n}{2}$ if $n$ is even;
     \end{itemize}
\item [(c)] The tensor product $E_{[0,  \cdots,  0, 1]}\otimes E_{[a_1, 0, \cdots, 0, a_{n-1},0]}$ is multiplicity-free and consists of
         $\widetilde{E}_{[a_1, 0, \cdots, 0, a_{n-1},0] + \nu}$, where $\nu$ runs over
        \begin{itemize}
         \item[$\bullet$]   $\varpi_{n}$, $-\varpi_{1}+\varpi_{n-1}$, $-\varpi_{1}+\varpi_{n-2i}-\varpi_{n-1}$ for $1\leq i < \frac{n}{2}$; $\varpi_{n-2i+1}-\varpi_{n-1}$ for $1< i \leq \frac{n}{2}$ if $n$ is odd;
         \item[$\bullet$]  $\varpi_{n}$, $-\varpi_{1}+\varpi_{n-1}$, $-\varpi_{1}+\varpi_{n-2i}-\varpi_{n-1}$ for $1\leq i <\frac{n}{2}$; $\varpi_{n-2i+1}-\varpi_{n-1}$ for $1< i \le \frac{n}{2}$ if $n$ is even.
     \end{itemize}
    \end{itemize}
\end{thm}
\begin{proof}
 We only prove case (b) for $n$ odd, other cases are similar. Take $\lambda^\prime=[0, \cdots, 0, 1, 0]$
 and $\lambda^{\prime\prime}=[a_1, 0, \cdots, 0, a_{n-1}, 0]$. We use Lemma 3.1 to compute $E_{\lambda^\prime}\otimes E_{\lambda^{\prime\prime}}$

 Firstly, note that $\lambda^\prime=\left(\frac{1}{2}, \dots,  \frac{1}{2}, -\frac{1}{2}\right)$. This gives us the extreme weights
\begin{equation}\label{Dn-extreme-weights}
 \left(\pm\frac{1}{2}, \pm\frac{1}{2}, \cdots, \pm\frac{1}{2}\right),
\end{equation}
 with $-\frac{1}{2}$ occurring odd times. It is easy to see that there are $2^{n-1}$ such weights since $$\sum_{i \,\rm{odd}}\binom{n}{i}=\frac{1}{2}\sum_{i}\binom{n}{i}=2^{n-1}\,.$$
 On the other hand, by the Weyl dimension formula, we compute that $\dim E_{\lambda^\prime}=2^{n-1}$. Thus we conclude that $\Pi(\lambda^\prime)$ consists of the above extreme weights, each occurring once.

 Now let us find out those $\lambda$s in \eqref{Dn-extreme-weights} such that $\lambda+\lambda^{\prime\prime}+\delta$ is regular.
 Indeed, $\lambda +\lambda^{\prime\prime}+\delta$ equals
 $$
 \left(n-1+a_1+\frac{a_{n-1}}{2}\pm\frac{1}{2}, n-2+\frac{a_{n-1}}{2}\pm\frac{1}{2}, \dots, 1+\frac{a_{n-1}}{2}\pm\frac{1}{2}, -\frac{a_{n-1}}{2}\pm\frac{1}{2}\right).
 $$
Since $a_1, a_{n-1}$ are non-negative integers, it follows that $\lambda+\lambda^{\prime\prime}+\delta$ is always dominant.
Denote $\lambda$ by $(\lambda_1, \lambda_2, \dots, \lambda_n)$ and denote $\lambda +\lambda^{\prime\prime}+\delta$ by $(\mu_1, \mu_2, \dots, \mu_n)$. For $\lambda+\lambda^{\prime\prime}+\delta$ to be regular, we should require that
    \begin{itemize}
    \item[$\bullet$] $\mu_i > \mu_{i+1}$ for any $1 \leq i \leq n-1$;
    \item[$\bullet$] $\mu_{n-1}+\mu_n\neq 0$.
    \end{itemize}
The first requirement excludes the cases where
  $\lambda_i=-\frac{1}{2}$ while $\lambda_{i+1}=\frac{1}{2}$
    for any $1 < i \leq n-1$. The second requirement rules out the cases where $\lambda_{n-1}=\lambda_n=-\frac{1}{2}$. Now the surviving  $\lambda $s are exactly $\varpi_{n-1}$, $-\varpi_{1}+\varpi_{n}$, $\varpi_{n-2i}-\varpi_{n-1}$ for $1\leq i < \frac{n}{2}$; $-\varpi_{1}+\varpi_{n-2i+1}-\varpi_{n-1}$ for $1< i \leq \frac{n}{2}$. This finishes the proof.
\end{proof}

For future use, let us specialize the above theorem for $D_5$ as follows.

\begin{cor}\label{cor-branching-D5}
Let $\frg$ be type $D_5$, and let $a, d\in\bbN$. Then we have that
\begin{itemize}
    \item [(a)] The tensor product $E_{[1,0,0,0,0]}\otimes E_{[a,0,0,d,0]}$ is multiplicity-free and consists of $\widetilde{E}_{[a+1, 0, 0, d, 0]}$, $\widetilde{E}_{[a-1, 0, 0, d, 0]}$,  $\widetilde{E}_{[a, 0, 0, d-1, 1]}$, $\widetilde{E}_{[a-1, 1, 0, d-1, 0]}$.
     \item [(b)] The tensor product $E_{[0,0,0,1,0]}\otimes E_{[a,0,0,d,0]}$ is multiplicity-free and consists of $\widetilde{E}_{[a, 0, 0, d+1, 0]}, $ $\widetilde{E}_{[a-1, 0, 0, d, 1]}, \widetilde{E}_{[a, 0, 1, d-1, 0]}, \widetilde{E}_{[a+1, 0, 0, d-1, 0]}, \widetilde{E}_{[a-1, 1, 0, d-1, 0]}$.
     \item [(c)] The tensor product $E_{[0,0,0,0,1]}\otimes E_{[a,0,0,d,0]}$ is multiplicity-free and consists of $\widetilde{E}_{[a, 0, 0, d, 1]}, $ $\widetilde{E}_{[a, 1, 0, d-1, 0]}, \widetilde{E}_{[a-1, 0, 0, d+1, 0]}, \widetilde{E}_{[a-1, 0, 1, d-1, 0]}, \widetilde{E}_{[a, 0, 0, d-1, 0]}$.
\end{itemize}
\end{cor}

\subsection{Type $E_6$}\label{sec-E6-tensor}

This section aims to prepare some technical lemmas concerning the tensor product decompositions of finite-dimensional modules of $E_6$, whose Dynkin diagram is presented in Figure \ref{Fig-E6-Dynkin}. We should note that using the involution of $E_6$ which exchanges  $\alpha_1$ and $\alpha_6$,  $\alpha_3$ and $\alpha_5$, while preserves $\alpha_2$ and $\alpha_4$, one can deduce Lemma \ref{lemma-E6-000001} from Lemma \ref{lemma-E6-100000}, deduce Lemma \ref{lemma-E6-000010} from Lemma \ref{lemma-E6-001000}, and deduce Lemma \ref{lemma-E6-000002} from Lemma \ref{lemma-E6-200000}. But we still present the latter all these lemmas for two reasons: presenting them clearly will offer convenience since we need to refer to them frequently later; they will allow us to check their dual versions and vice versa since they are deduced independently.
%\begin{figure}[H]
%\centering
%\scalebox{0.6}{\includegraphics{E6-Dynkin.eps}}
%\caption{The Dynkin diagram for $E_6$}
%\label{Fig-E6-Dynkin}
%\end{figure}
\begin{figure}[htb]
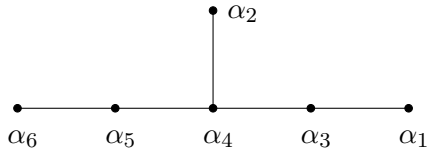

\begin{center}
\begin{dynkinDiagram}[text style/.style={scale=1},edge length=1.3cm,labels={6,2,5,4,3,1},label macro/.code={\alpha_{\drlap{#1}}}]E6
\end{dynkinDiagram}
\caption{The Dynkin diagram for $E_6$}\label{Fig-E6-Dynkin}
\end{center}
\end{figure}

\begin{lemma}\label{lemma-E6-100000}
Let $\frg$ be $E_6$.
For any $a, f\in\bbN$, the tensor product $E_{[1, 0, 0, 0, 0, 0]}\otimes E_{[a,0,0,0,0,f]}$ is multiplicity-free, and it  consists of the following components:
$$
\widetilde{E}_{[a+1, 0, 0, 0, 0, f]}, \widetilde{E}_{[a-1, 0, 1, 0, 0, f]}, \widetilde{E}_{[a, 1, 0, 0, 0, f-1]},
\widetilde{E}_{[a-1, 0, 0, 0, 0, f+1]}, \widetilde{E}_{[a-1, 0, 0, 0, 1, f-1]}, \widetilde{E}_{[a, 0, 0, 0, 0, f-1]}.
$$
\end{lemma}
\begin{proof}
Using the Weyl dimension formula, one computes that $E_{[1, 0, 0, 0, 0, 0]}$ has dimension $27$. On the other hand, one computes that $E_{[1, 0, 0, 0, 0, 0]}$ has exactly $27$ extreme weights, which are listed below:
\begin{eqnarray*}
&[1, 0, 0, 0, 0, 0], \, [-1, 0, 1, 0, 0, 0], \, [0, 0, -1, 1, 0, 0], \, [0, 1, 0, -1, 1, 0], \, [0, -1, 0, 0, 1, 0],\\
&[0, 1, 0, 0, -1, 1], \, [0, -1, 0, 1, -1, 1], \, [0, 1, 0, 0, 0, -1], \, [0, 0, 1, -1, 0, 1], \, [0, -1, 0, 1, 0, -1], \\
&[1, 0, -1, 0, 0, 1], \, [0, 0, 1, -1, 1, -1], \, [-1, 0, 0, 0, 0, 1], \, [1, 0, -1, 0, 1, -1], \, [0, 0, 1, 0, -1, 0], \\
&[-1, 0, 0, 0, 1, -1], \, [1, 0, -1, 1, -1, 0], \, [-1, 0, 0, 1, -1, 0],\, [1, 1, 0, -1, 0, 0],\, [-1, 1, 1, -1, 0, 0], \\
&[1, -1, 0, 0, 0, 0], \, [-1, -1, 1, 0, 0, 0], \, [0, 1, -1, 0, 0, 0], \, [0, -1, -1, 1, 0, 0], \, [0, 0, 0, -1, 1, 0], \\
&[0, 0, 0, 0, -1, 1], \, [0, 0, 0, 0, 0, -1].
\end{eqnarray*}
Therefore, the module $E_{[1, 0, 0, 0, 0, 0]}$ is weight-free, and all of its weights are exactly the above ones. Moreover, each coordinate of such a weight $\lambda=[\lambda_1, \lambda_2, \dots, \lambda_6]$ is lower bounded by $-1$. Now take $\lambda^{\prime}=[1,0,0,0,0,0]$, and $\lambda^{\prime\prime}=[a,0,0,0,0,f]$. Note that $\delta=[1,1,1,1,1,1]$. Thus for any weight $\lambda$ of $E_{\lambda'}$,  the weight $\lambda^{\prime\prime}+\lambda+\delta$ must be dominant, and it is regular if and only if  $\lambda_2, \lambda_3, \lambda_4, \lambda_5\geq 0$. This selects out the following six $\lambda$s: $[1, 0, 0, 0, 0, 0]$, $[-1, 0, 1, 0, 0, 0]$, $[0, 1, 0, 0, 0, -1]$, $[-1, 0, 0, 0, 0, 1]$,  $[-1, 0, 0, 0, 1, -1]$,  $[0, 0, 0, 0, 0, -1]$.
The desired result now follows from Lemma \ref{lemma-tensor-prod}.
\end{proof}

We shall also need the following lemmas, whose proof are similar to the previous one's and are thus omitted.

\begin{lemma}\label{lemma-E6-000001}
For any $a, f\in\bbN$, the tensor product $E_{[0, 0, 0, 0, 0, 1]}\otimes E_{[a,0,0,0,0,f]}$ is multiplicity-free, and it  consists of the following components:
$$
\widetilde{E}_{[a, 0, 0, 0, 0, f+1]}, \widetilde{E}_{[a, 0, 0, 0, 1, f-1]}, \widetilde{E}_{[a-1, 1, 0, 0, 0, f]},
\widetilde{E}_{[a+1, 0, 0, 0, 0, f-1]}, \widetilde{E}_{[a-1, 0, 1, 0, 0, f-1]}, \widetilde{E}_{[a-1, 0, 0, 0, 0, f]}.
$$
\end{lemma}

The following lemmas are still proven via Lemma \ref{lemma-tensor-prod}.
However, the situations will be more complicated. For instance, we will meet weight-multiplicities and sometimes the cases $a=0$ or $f=0$ need to be treated separately.  We  provide some details for Lemma \ref{lemma-E6-100001}.

\begin{lemma}\label{lemma-E6-010000}
For any $a, f\in\bbN$, the tensor product $E_{[0,  1, 0, 0,  0, 0]}\otimes E_{[a,0,0,0,0,f]}$ consists of the components $\widetilde{E}_{[a,0,0,0,0,f]}+\nu$, where $\nu$ runs over
\begin{eqnarray*}
&[-2, 0, 1, 0, 0, 0], [0, 0, 0, 0, 1, -2], [0, 1, 0, 0, 0, 0], [-1, 0, 0, 0, 1, 0], \\
&[0, 0, 1, 0, 0, -1], [-1, 0, 0, 1, 0, -1], [-1, 1, 0, 0, 0, -1], [0, 0, 0, 0, 0, 0]^{\underline{2}}.
\end{eqnarray*}
Here $\nu^{\underline{p}}$ means that $\nu$ occurs with multiplicity $p$.
Moreover, when $a=0$ or $f=0$, we should further subtract $\nu=[0, 0, 0, 0, 0, 0]$.
\end{lemma}

\begin{lemma}\label{lemma-E6-001000}
For any $a, f\in\bbN$, the tensor product $E_{[0, 0, 1, 0, 0, 0]}\otimes E_{[a,0,0,0,0,f]}$ consists of the components $\widetilde{E}_{[a,0,0,0,0,f]}+\nu$, where $\nu$ runs over
\begin{eqnarray*}
&[-2, 0, 1, 0, 0, 1], [-2, 0, 1, 0, 1, -1], [-2, 0, 0, 0, 1, 0], [-2,  0, 0, 1, 0, -1], [0, 0, 0, 1, 0, -2],\\
&[-1, 1, 0, 0, 1, -2], [0, 1,   0, 0, 0, -2], [-1, 0, 0, 0, 1, -2], [0, 0, 1, 0, 0, 0], [-1, 0, 0, 1, 0, 0], \\
&[1, 1, 0, 0, 0, -1], [-1, 1, 1, 0, 0, -1], [0, 0, 0, 0, 0, 1], [1, 0, 0, 0, 0, -1], [-1, 0, 0, 0, 0, 0],\\
&[-1, 0, 1, 0, 0, -1]^{\underline{2}},  [0, 0, 0, 0, 1, -1]^{\underline{2}},  [-1, 1, 0, 0, 0, 0]^{\underline{2}}.
\end{eqnarray*}
Moreover, when $a=0$, we should further subtract the following $\nu$s
$$
[0, 0, 0, 0, 0, 1], [0, 0, 0, 0, 1, -1];
$$
while when $f=0$, we should further subtract the following $\nu$s
$$
[-1, 1, 0, 0, 0, 0], [-1, 0, 0, 0, 0, 0].
$$
\end{lemma}

\begin{lemma}\label{lemma-E6-000010}
For any $a, f\in\bbN$, the tensor product $E_{[0, 0, 0, 0, 1, 0]}\otimes E_{[a,0,0,0,0,f]}$ consists of the components $\widetilde{E}_{[a,0,0,0,0,f]}+\nu$, where $\nu$ runs over
\begin{eqnarray*}
&[-2,0,0,1,0,0], [-2,1,1,0,0,-1], [-2,1,0,0,0,0], [-2,0,1,0,0,-1], [1,0,0,0,1,-2], \\
&[-1,0,1,0,1,-2], [0,0,1,0,0,-2], [-1,0,0,1,0,-2], [0,0,0,0,1,0], [0,0,0,1,0,-1],\\
&[-1,1,0,0,0,1], [-1,1,0,0,1,-1], [1,0,0,0,0,0], [0,0,0,0,0,-1], [-1,0,0,0,0,1], \\ &[-1,0,1,0,0,0]^{\underline{2}}, [0,1,0,0,0,-1]^{\underline{2}},  [-1,0,0,0,1,-1]^{\underline{2}}, .
\end{eqnarray*}
Moreover, when $a=0$, we should further subtract the following $\nu$s
$$
[0, 1, 0, 0, 0, -1], [0, 0, 0, 0, 0, -1];
$$
while when $f=0$, we should further subtract the following $\nu$s
$$
[1, 0, 0, 0, 0, 0], [-1, 0, 1, 0, 0, 0].
$$
\end{lemma}

\begin{lemma}\label{lemma-E6-200000}
For any $a, f\in\bbN$, the tensor product $E_{[2, 0, 0, 0, 0, 0]}\otimes E_{[a,0,0,0,0,f]}$ consists of the components $\widetilde{E}_{[a,0,0,0,0,f]}+\nu$, where $\nu$ runs over
\begin{eqnarray*}
&[-2,0,2,0,0,0], [-2,0,1,0,0,1], [-2,0,1,0,1,-1], [-2,0,0,0,0,2], [-2,0,0,0,1,0], \\
&[0,2,0,0,0,-2], [-1,1,0,0,1,-2], [0,1,0,0,0,-2], [-1,0,0,0,1,-2], [0,0,0,0,0,-2], \\
&[2,0,0,0,0,0], [0,0,1,0,0,0], [1,1,0,0,0,-1], [-1,1,1,0,0,-1], [0,0,0,0,1,-1],\\
&[-1,1,0,0,0,0], [-1,0,1,0,0,-1], [0,0,0,0,0,1], [-2,0,0,0,2,-2], [1,0,0,0,0,-1], \\
&[-1,0,0,0,0,0].
\end{eqnarray*}
Moreover, when $a=0$, we should further subtract the following $\nu$s
$$
[0, 0, 1, 0, 0, 0], [0, 0, 0, 0, 0, 1], [0, 0, 0, 0, 1, -1];
$$
while when $f=0$, we should further subtract the following $\nu$s
$$
[-1, 1, 0, 0, 0, 0], [-1, 0, 0, 0, 0, 0], [-2, 0, 0, 0, 1, 0].
$$
\end{lemma}

\begin{lemma}\label{lemma-E6-000002}
For any $a, f\in\bbN$, the tensor product $E_{[0, 0, 0, 0, 0, 2]}\otimes E_{[a,0,0,0,0,f]}$ consists of the components $\widetilde{E}_{[a,0,0,0,0,f]}+\nu$, where $\nu$ runs over
\begin{eqnarray*}
&[-2,2,0,0,0,0], [-2,1,1,0,0,-1], [-2,1,0,0,0,0], [-2,0,1,0,0,-1], [-2,0,0,0,0,0],\\
&[0,0,0,0,2,-2], [1,0,0,0,1,-2], [-1,0,1,0,1,-2], [2,0,0,0,0,-2], [0,0,1,0,0,-2], \\
&[0,0,0,0,0,2], [0,0,0,0,1,0], [-1,1,0,0,0,1], [-1,1,0,0,1,-1], [-1,0,1,0,0,0],\\
&[0,1,0,0,0,-1], [-1,0,0,0,1,-1], [1,0,0,0,0,0], [-2,0,2,0,0,-2], [-1,0,0,0,0,1], \\
&[0,0,0,0,0,-1].
\end{eqnarray*}
Moreover, when $a=0$, we should further subtract the following $\nu$s
$$
[0, 1, 0, 0, 0, -1], [0, 0, 0, 0, 0, -1], [0, 0, 1, 0, 0, -2];
$$
while when $f=0$, we should further subtract the following $\nu$s
$$
[0, 0, 0, 0, 1, 0], [1, 0, 0, 0, 0, 0], [-1, 0, 1, 0, 0, 0].
$$
\end{lemma}

\begin{lemma}\label{lemma-E6-100001}
For any $a, f\in\bbN$, the tensor product $E_{[1, 0, 0, 0, 0, 1]}\otimes E_{[a,0,0,0,0,f]}$ consists of the components $\widetilde{E}_{[a,0,0,0,0,f]}+\nu$, where $\nu$ runs over
\begin{eqnarray*}
&[-2,1,1,0,0,0],  [-2,0,2,0,0,-1], [-2,1,0,0,0,1],  [-2,1,0,0,1,-1],  [-2,0,0,0,0,1],\\
&[-2,0,0,0,1,-1], [0,1,0,0,1,-2],  [-1,0,0,0,2,-2], [1,1,0,0,0,-2],   [-1,1,1,0,0,-2], \\
&[1,0,0,0,0,-2],  [-1,0,1,0,0,-2],  [1,0,0,0,0,1],  [-1,0,1,0,0,1],   [1,0,0,0,1,-1], \\
&[-1,0,1,0,1,-1], [-1,0,0,0,0,2],   [2,0,0,0,0,-1], [-1,2,0,0,0,-1],  [-2,0,1,0,1,-2],  \\
&[-1,0,0,0,0,-1], [-2,0,1,0,0,0]^{\underline{2}},  [0,0,0,0,1,-2]^{\underline{2}},  [0,1,0,0,0,0]^{\underline{2}}, [-1,0,0,0,1,0]^{\underline{2}}, \\
&[0,0,1,0,0,-1]^{\underline{2}}, [-1,0,0,1,0,-1]^{\underline{2}}, [-1,1,0,0,0,-1]^{\underline{2}},  [0,0,0,0,0,0]^{\underline{3}}.
\end{eqnarray*}
Moreover, when $a=0$, we should further subtract the following $\nu$s
$$
[0, 1, 0, 0, 0, 0], [0, 0, 1, 0, 0, -1], [0, 0, 0, 0, 1, -2], [0, 0, 0, 0, 0, 0]^{\underline{2}};
$$
while when $f=0$, we should further subtract the following $\nu$s
$$
[0, 1, 0, 0, 0, 0], [-1, 0, 0, 0, 1, 0], [-2, 0, 1, 0, 0, 0], [0, 0, 0, 0, 0, 0]^{\underline{2}}.
$$
\end{lemma}
\begin{proof}
Take $\lambda^\prime=[1,0,0,0,0,1]$ and $\lambda^{\prime\prime}=[a,0,0,0,0,f]$. One computes that $\dim E_{\lambda'}=650$.
The lemma is deduced by running the RHS of \eqref{tensor-prod-summands} over the $650$ weights in $\Pi(\lambda^\prime)$.
For instance, take $\lambda_1=[0, -2, 0, 2, -2, 1]\in \Pi(\lambda^\prime)$. Then $\lambda_1+\delta=[1,-1,1,3,-1,2]$, which will be mapped to
$[1,1,1,1,1,1]$ by $s_{\alpha_2}s_{\alpha_5}$. We see that $\lambda_1$ contributes a copy of $\textbf{+}[0,0,0,0,0,0]$.

Now take $\lambda_2=[0, -2, 0, 1, 1, -2]$. Then $\lambda_2+\delta=[1,-1,1,2,2,-1]$, which will be mapped to
$[1,1,1,1,2,-1]$ by $s_{\alpha_2}$. We see that $\lambda_2$ contributes a copy of $\textbf{-}[0,0,0,0,1,-2]$. Moreover, when $f=0$,
$[1,1,1,1,2,-1]$ will be further mapped to $[1,1,1,1,1,1]$ by $s_{\alpha_1}$, meaning that $\lambda_2$ also contributes a copy of $\textbf{+}[0,0,0,0,0,0]$ when $f=0$.

Going through all the $650$ weights of $\Pi(\lambda^\prime)$, computing their contributions one by one, and then collecting the signed terms proves the lemma.
\end{proof}

For the lemmas in this subsection, we do not write the case $a=0$ and $f=0$ since the module $E_{[a,0,0,0,0,f]}$ becomes trivial.
Another remark is that \texttt{SageMath}  \cite{SAGE24} has been employed to verify the above lemmas on various examples.
Let us provide one such verification.

\begin{example}
We obtain the tensor product decomposition of $E_{[1, 0, 0, 0, 0, 1]}\otimes E_{[2,0,0,0,0,2]}$ by typing the following command:
\begin{verbatim}
sage: E6 = WeylCharacterRing("E6",style="coroots")
sage: E6(1,0,0,0,0,1)*E6(2,0,0,0,0,2)
\end{verbatim}
The output is as follows:
\begin{verbatim}
E6(0,0,2,0,0,1) + E6(1,0,1,0,0,0) + 2*E6(2,0,1,0,0,1) + 2*E6(0,0,1,0,0,2) +
E6(1,0,1,0,0,3) + E6(0,0,1,0,1,0) + E6(1,0,1,0,1,1) + E6(3,1,0,0,0,0) +
2*E6(1,1,0,0,0,1) + 2*E6(2,1,0,0,0,2) + E6(0,1,0,0,0,3) + E6(2,1,0,0,1,0) +
E6(0,1,0,0,1,1) + E6(1,2,0,0,0,1) + E6(1,1,1,0,0,0) + E6(0,1,1,0,0,2) +
2*E6(1,0,0,1,0,1) + 2*E6(2,0,0,0,1,0) + E6(3,0,0,0,1,1) + E6(0,0,0,0,1,1) +
2*E6(1,0,0,0,1,2) + E6(1,0,0,0,2,0) + E6(4,0,0,0,0,1) + E6(1,0,0,0,0,1) +
3*E6(2,0,0,0,0,2) + E6(3,0,0,0,0,3) + E6(0,0,0,0,0,3) + E6(1,0,0,0,0,4) +
E6(3,0,0,0,0,0)
\end{verbatim}
This agrees with Lemma \ref{lemma-E6-100001} for $a=f=2$.%\hfill\qed
\end{example}

\section{The group $E_{6(-14)}$}\label{sec-E6}

In this section, we fix $G$ as $E_{6(-14)}$, which is a connected equal rank linear group with center $\bbZ/3\bbZ$. Its Lie algebra $\frg_0$ of $G$ is labelled as EIII in Appendix C of Knapp \cite{Kn}, the corresponding Vogan diagram is shown in Figure \ref{Fig-E614-Vogan}.

\begin{figure}[htb]
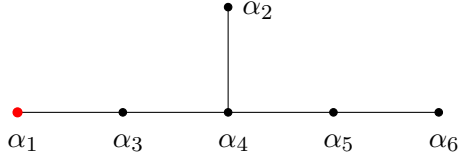

\begin{center}
\begin{dynkinDiagram}[text style/.style={scale=1},edge length=1.4cm,labels={1,2,3,4,5,6},label macro/.code={\alpha_{\drlap{#1}}}]E6
\fill[red,draw=red] (root 1) circle (.06cm);
\end{dynkinDiagram}
\caption{The Vogan diagram for $E_{6(-14)}$}\label{Fig-E614-Vogan}
\end{center}
\end{figure}

As in the introduction, $T$ is a maximal torus of $K$.
We fix a choice of $\Delta^+(\frg, \frt)$ by specifying the following roots as the simple ones:
\begin{align*}
\alpha_1 &=\frac{1}{2}(1, -1,-1,-1,-1,-1,-1,1) \,,\\
\alpha_2 &=e_1+e_2 \,,\; \text{and}\; \alpha_i=e_{i-1}-e_{i-2} \, \text{ for }3\leq i\leq 6 \,.
\end{align*}
%$\alpha_1=\frac{1}{2}(1, -1,-1,-1,-1,-1,-1,1)$, $\alpha_2=e_1+e_2$ and $\alpha_i=e_{i-1}-e_{i-2}$ for $3\leq i\leq 6$.
Among them, only $\alpha_1$ is \emph{non-compact}. That is, $\frg_{\alpha_1} \subset \frp$. Correspondingly, the fundamental weights are denoted by $\zeta_1, \zeta_2, \dots, \zeta_6$. The Lie algebra $\frk$ has a one-dimensional center which is spanned by $\zeta:=\zeta_1=(0,0,0,0,0, -\frac{2}{3}, -\frac{2}{3}, \frac{2}{3})$.
Note that $\Delta^+(\frg, \frt)=\Delta^+_{\frk} \cup \Delta^+_n$, that $\delta=(0, 1, 2, 3, 4, -4, -4, 4)$, $\delta_c=(0, 1, 2, 3, 4, 0, 0, 0)$ and that $\delta_n=6\zeta$.

The simple roots for the above $\Delta^+_{\frk}$ are $\gamma_1:=\alpha_2$, $\gamma_2:=\alpha_3$, $\gamma_3:=\alpha_4$, $\gamma_4:=\alpha_5$ and $\gamma_5:=\alpha_6$.  Let $\varpi_1, \varpi_2, \dots, \varpi_5$ be the corresponding fundamental weights. The simple roots $\gamma_i$, $1\leq i\leq 5$, span the orthogonal complement $\frt^*_{-}$ of $\bbC \zeta$ in $\frt^*$. Note that $\Delta^+_{\frk}$ is of type $D_5$, but its labelling of simple roots is \textbf{reverse}  to that in Section \ref{sec-typeD}. Indeed, the Dynkin diagram for $\Delta^+(\frk, \frt_{-})$ is as in Figure \ref{Fig-EIII-Dynkin}.

%\begin{figure}[H]
%\centering
%\scalebox{0.6}{\includegraphics{EIII-K-Dynkin.eps}}
%\caption{The Dynkin diagram for $\Delta^+(\frk, \frt_{-})$}
%\label{Fig-EIII-Dynkin}
%\end{figure}

\begin{figure}[htb]
\begin{center}
\begin{dynkinDiagram}[text style/.style={scale=1},backwards,edge length=1.6cm,labels={5,4,3,1,2},label macro/.code={\gamma_{\drlap{#1}}}]D5
\end{dynkinDiagram}
\caption{The Dynkin diagram for $\Delta^+(\frk, \frt_{-})$}\label{Fig-EIII-Dynkin}
\end{center}
\end{figure}

As in \cite{DDH}, we use $\{\varpi_1, \dots, \varpi_5, \frac{\zeta}{4}\}$ as a basis to express the highest weight of a $\frk$-type. In other words, the six-tuple $[n_1, \dots, n_5, n_6]$ stands for the vector $n_1\varpi_1 + \cdots + n_5\varpi_5 + \frac{n_6}{4} \zeta$. For instance, $\delta_n=[0,0,0,0,0,24]$.
For $E_{6(-14)}$, the basic Schmid modules for $S(\frp^-)$ are $[0,1,0,0,0,-3]$ and $[0,0,0,0,1,-6]$. See \cite{S} or Table 2 of \cite{PPSV}.   Therefore, we have that
\begin{equation}\label{p-E6}
S(\frp^-)=\bigoplus_{b, e\in\bbN} E_{[0,b,0,0,e, -3b-6e]}
\end{equation}
as $\frk$-modules. Here we emphasize that we should \textbf{reverse} the first five coordinates of a weight $[n_1, n_2, \dots, n_6]$ to match the root system of $D_5$ in Section \ref{sec-typeD}.
Moreover, $\widetilde{E}_{[n_1, n_2, \dots, n_6]}$ stands for a $\frk$-module if and only if $n_1, n_2, \dots, n_5\geq 0$, while in this case $n_6$ \textbf{can be negative}.

Now take $\pi$ to be the Wallach module $L(-3\zeta)$. Recall from \cite{EHW} that $z=-3$ is the first reduction point of the modules $L(z\zeta)$ for $z\in\bbZ_{\leq 0}$. According to Example 6.3 of \cite{DDH}, we have that
$$
H_D^+(\pi)=E_{[0,0,0,0,0,12]}\oplus E_{[1,0,0,0,0,3]}\oplus E_{[0,0,0,0,1,-6]},
$$
and that
$$
H_D^-(\pi)=E_{[0,0,0,0,0,-12]}\oplus E_{[0,1,0,0,0,-3]}\oplus E_{[0,0,0,0,1,6]}.
$$
By Corollary \ref{cor-branching-D5} and \eqref{p-E6}, we compute all the $K$-types of $N(\xi)$ and $N(\eta)$, where $\xi$ (resp., $\eta$) runs over the $\widetilde{K}$-types in $H_D^+(\pi)$ (resp., $H_D^-(\pi)$). They are listed in Table \ref{table-E6-Wallach}. Now let us carry out the cancellations as required by \eqref{eq:HPZthmA}.
Indeed, from the second row to the last row of Table \ref{table-E6-Wallach}, the left hand side (LHS for short) completely cancels out with the right hand side (RHS for short). When one comes to the first row, what remains after the cancellation are $[0,b,0,0,e,-3b-6e+12]$ for $e=0$ and $b\in\bbN$. Finally, remembering the $\delta_n$ shift in \eqref{eq:HPZthmA}, we obtain the following.

\begin{thm}\label{thm-K-E6-Wallach}
We have that
\begin{equation}\label{K-E6-Wallach}
L(-3\zeta)|_K \cong \bigoplus_{b\in\bbN} E_{[0,b,0,0,0,-3b-12]}.
\end{equation}
\end{thm}

It is worth mentioning that to compute the Dirac cohomology of $L(-3\zeta)$, as required by Theorem \ref{thm-HP}, it suffices to look at its $K$-types up to the \texttt{atlas} (cf. \cite{ALTV,At})  height $114$. Indeed,
\begin{verbatim}
atlas> G:E6_h
atlas> set p=parameter(G,496,[1,3,1,0,1,1],[0,4,0,-1,1,1])
atlas> print_branch_irr_long(p,KGB(G,26),114)
m  x   lambda                  hw                     dim  height
1  15  [ 0,0,0,0,1,1]/1  KGB element #26[0,0,0,0,0,3]  1    46
1  55  [ 0,1,0,0,1,1]/1  KGB element #26[0,1,0,0,0,3]  16   67
1  18  [ 0,0,0,1,1,1]/1  KGB element #26[0,2,0,0,0,3]  126  88
1  18  [ 0,1,0,1,1,1]/1  KGB element #26[0,3,0,0,0,3]  672  110
\end{verbatim}
These four $K$-types turn out to be all the spin LKTs of $L(-3\zeta)$. It is from these four $K$-types that we recover the whole $K$-specturm of $L(-3\zeta)$.

Now let us consider the highest weight module $L(\mu)$, where $\mu=(0,0,0,0,1,3,3,-3)$. As computed in Example 6.1 of \cite{DDH}, we have that
$$
H_D^+(L(\mu))=E_{[0,0,0,0,1,6]}  \oplus E_{[0,1,0,0,0,-3]} \oplus E_{[0,0,0,0,0,-12]},
$$
and that
$$
H_D^-(L(\mu))=E_{[0,0,0,0,1,-6]}  \oplus E_{[1,0,0,0,0,-3]}.
$$

By Corollary \ref{cor-branching-D5} and \eqref{p-E6}, we compute all the $K$-types of $N(\xi)$ and $N(\eta)$, where $\xi$ (resp., $\eta$) runs over the $\widetilde{K}$-types in $H_D^+(L(\mu))$ (resp., $H_D^-(L(\mu))$). They are listed in Table \ref{table-E6-non-Wallach}. Now let us carry out the cancellations as required by \eqref{eq:HPZthmA}.
Indeed, from the first row to the ninth row of Table \ref{table-E6-non-Wallach}, the LHS completely cancels out with the RHS. Only the LHS of the last row remains. Finally, remembering the $\delta_n$ shift in \eqref{eq:HPZthmA}, we obtain the following result, which agrees with Corollary 12.6 of \cite{EHW}.

\begin{thm}\label{thm-K-E6-non-Wallach}
Put $\mu=(0,0,0,0,1,3,3,-3)$, we have that
\begin{equation}\label{K-E6-non-Wallach}
L(\mu)|_K \cong \bigoplus_{b, e\in\bbN} E_{[0,b,0,0,e+1,-3b-6e-18]}.
\end{equation}
\end{thm}

\begin{table}
\caption{Involved $K$-types for the Wallach module $L(-3\zeta)$ of $E_{6(-14)}$}
\begin{tabular}{cc}
the $K$-types of $N(\xi)$  &  the $K$-types of $N(\eta)$  \\
\hline
$*[0,b,0,0,e,-3b-6e+12]$  &  $[0,b,0,0,e+1,-3b-6e+6]$\\
$[1,b,0,0,e,-3b-6e+3]$ & $[1,b-1,0,0,e,-3b-6e+6]$  \\
$[0,b+1,0,0,e-1,-3b-6e+3]$  &  $[0,b+1,0,0,e,-3b-6e-3]$ \\
$[0,b-1,1,0,e-1,-3b-6e+3]$ & $[0,b-1,1,0,e,-3b-6e-3]$\\
$[0,b-1,0,1,e,-3b-6e+3]$  & $[0,b,0,1,e-1,-3b-6e+6]$\\
$[0,b-1,0,0,e,-3b-6e+3]$ &  $[0,b,0,0,e-1,-3b-6e+6]$ \\
$[0,b,0,0,e+1,-3b-6e-6]$ & $[0,b-1,0,0,e+1,-3b-6e-3]$\\
$[0,b,0,0,e-1,-3b-6e-6]$ & $[0,b,0,0,e,-3b-6e-12]$\\
$[1,b-1,0,0,e,-3b-6e-6]$  &  $[1,b,0,0,e-1,-3b-6e-3]$\\
$[0,b,0,1,e-1,-3b-6e-6]$ & $[0,b-1,0,1,e-1,-3b-6e-3]$\\
\hline
\end{tabular}
\label{table-E6-Wallach}
\end{table}

\begin{table}
\caption{Involved $K$-types for the module $L(\mu)$ of $E_{6(-14)}$}
\begin{tabular}{cc}
the $K$-types of $N(\xi)$  &  the $K$-types of $N(\eta)$  \\
\hline
$[0,b,0,0,e,-3b-6e-12]$  &  $[0,b,0,0,e-1,-3b-6e-6]$\\
$[0,b,0,0,e-1,-3b-6e+6]$  &  $[0,b-1,0,0,e,-3b-6e+3]$ \\
$[0,b,0,1,e-1,-3b-6e+6]$ & $[0,b-1,0,1,e,-3b-6e+3]$\\
$[1,b-1,0,0,e,-3b-6e+6]$  & $[1,b,0,0,e,-3b-6e+3]$\\
$[0,b+1,0,0,e,-3b-6e-3]$ &  $[0,b+1,0,0,e-1,-3b-6e+3]$ \\
$[1,b,0,0,e-1,-3b-6e-3]$ & $[1,b-1,0,0,e,-3b-6e-6]$\\
$[0,b-1,1,0,e,-3b-6e-3]$ & $[0,b-1,1,0,e-1,-3b-6e+3]$\\
$[0,b-1,0,0,e+1,-3b-6e-3]$  &  $[0,b,0,0,e+1,-3b-6e-6]$\\
$[0,b-1,0,1,e-1,-3b-6e-3]$ & $[0,b,0,1,e-1,-3b-6e-6]$\\
$*[0,b,0,0,e+1,-3b-6e+6]$ &   \\

\hline
\end{tabular}
\label{table-E6-non-Wallach}
\end{table}

\section{The group $E_{7(-25)}$}\label{sec-E7}

Now we fix $G$ as $E_{7(-25)}$, which is a connected equal rank linear group with center $\bbZ/2\bbZ$. Its Lie algebra $\frg_0$ of $G$ is labelled as EVII in Appendix C of Knapp \cite{Kn}, the corresponding Vogan diagram is shown in Figure \ref{Fig-E725-Vogan}.
\begin{figure}[htb]
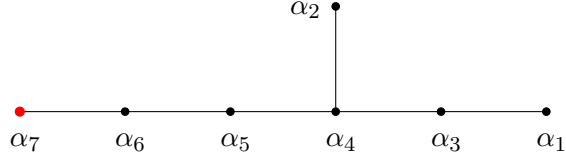

\begin{center}
\begin{dynkinDiagram}[text style/.style={scale=1},backwards,edge length=1.4cm,labels={1,2,3,4,5,6,7},label macro/.code={\alpha_{\drlap{#1}}}]E7
\fill[red,draw=red] (root 7) circle (.06cm);
\end{dynkinDiagram}
\caption{The Vogan diagram for $E_{7(-25)}$}\label{Fig-E725-Vogan}
\end{center}
\end{figure}

We fix a choice of $\Delta^+(\frg, \frt)$ by specifying the following roots as the simple ones:
\begin{align*}
\alpha_1 & =\frac{1}{2}(1, -1,-1,-1,-1,-1,-1,1) \,,\\
\alpha_2 & =e_1+e_2 \,,\; \text{and}\; \alpha_i =e_{i-1}-e_{i-2} \text{ for } 3 \leq i\leq 7 \,.
\end{align*}
%$\alpha_1=\frac{1}{2}(1, -1,-1,-1,-1,-1,-1,1)$, $\alpha_2=e_1+e_2$ and $\alpha_i=e_{i-1}-e_{i-2}$ for $3\leq i\leq 7$.
Among them, only $\alpha_7$ is \emph{non-compact}. That is, $\frg_{\alpha_7} \subset \frp$. Correspondingly, the fundamental weights are denoted by $\zeta_1, \zeta_2, \dots, \zeta_7$. The Lie algebra $\frk$ has a one-dimensional center which is spanned by $\zeta:=\zeta_7=\left(0, 0, 0, 0, 0, 1, -\frac{1}{2}, \frac{1}{2}\right)$.
Note that $\Delta^+(\frg, \frt)=\Delta^+_{\frk} \cup \Delta^+_n$, that $\delta=\left(0, 1, 2, 3, 4, 5, -\frac{17}{2}, \frac{17}{2}\right)$, $\delta_c=(0, 1, 2, 3, 4, -4, -4, 4)$ and that $\delta_n=9\zeta$.

The simple roots for the above $\Delta^+_{\frk}$ are $\gamma_i:=\alpha_i$ for $1\leq i\leq 6$.  Let $\varpi_1, \varpi_2, \dots, \varpi_6$ be the corresponding fundamental weights. The simple roots $\gamma_i$, $1\leq i\leq 6$, span the orthogonal complement $\frt^*_{-}$ of $\bbC \zeta$ in $\frt^*$. Note that $\Delta^+_{\frk}$ is of type $E_6$, and its labelling of simple roots is the same as that in Section \ref{sec-E6-tensor}.

As in \cite{DDH}, we use $\{\varpi_1, \dots, \varpi_6, \frac{\zeta}{3}\}$ as a basis to express the highest weight of a $\frk$-type. In other words, the seven-tuple $[n_1, \dots, n_6, n_7]$ stands for the vector $n_1\varpi_1 + \cdots + n_6\varpi_6 + \frac{n_7}{3} \zeta$. For instance, $\delta_n=[0,0,0,0,0,0,27]$. For $E_{7(-25)}$, the basic Schmid modules for $S(\frp^-)$ are $[0,0,0,0,0,0,-6]$, $[1,0,0,0,0,0,-4]$ and $[0,0,0,0,0,1,-2]$. See \cite{S} or Table 2 of \cite{PPSV}. Therefore, we have that
\begin{equation}
S(\frp^-)=\bigoplus_{a, f, n\in\bbN} E_{[a,0,0,0,0,f, -4a-2f-6n]}
\end{equation}
as $\frk$-modules.

Now we have two Wallach modules: $\pi_1=L(-4\zeta)$ and $\pi_2=L(-8\zeta)$. As computed in \cite{DD}, we have that
$H_D^+(\pi_1)$ consists of the following $\widetilde{K}$-types
$$
[0, 0, 0, 0, 0, 0, \pm 15], \quad [0, 1, 0, 0, 0, 0, \pm 9], \quad [1, 0, 0, 0, 0, 1, \pm 3],
$$
that
$H_D^-(\pi_1)$ consists of the following $\widetilde{K}$-types
\begin{eqnarray*}
&[1, 0, 0, 0, 0, 0, 11], \quad [0, 0, 0, 0, 0, 1, -11],
\quad [2, 0, 0, 0, 0, 0, 1], \quad [0, 0, 0, 0, 0, 2, -1], \\
&[0, 0, 0, 0, 1, 0, 5], \quad [0, 0, 1, 0, 0, 0, -5],
\end{eqnarray*}
and that
$$
H_D^+(\pi_2)=E_{[0,0,0,0,0,0,3]}, \quad H_D^-(\pi_2)=E_{[0,0,0,0,0,0,-3]}.
$$

\begin{thm}\label{thm-K-E7-Wallach}
We have that
\begin{equation}\label{K-E7-Wallach-1}
L(-4\zeta)|_K \cong \bigoplus_{f\in\bbN} E_{[0,0,0,0,0, f, -2f-12]}.
\end{equation}
and that
\begin{equation}\label{K-E7-Wallach-2}
L(-8\zeta)|_K \cong \bigoplus_{a,f\in\bbN} E_{[a, 0, 0, 0, 0, f, -4a-2f-24]}.
\end{equation}
\end{thm}
\begin{proof}
The formula \eqref{K-E7-Wallach-2} is easy to deduce. We focus on \eqref{K-E7-Wallach-1}.
Firstly, note that a $K$-type in the positive part can not cancel with a $K$-type in the negative part if they differ from each other in certain coordinates. This hints us to organize all the involved $K$-types according to their second, third, fourth and fifth coordinates.

For instance, Table \ref{table-E7-Wallach-0000} collects all the $K$-types $\mu=[\mu_1, \dots, \mu_7]$ such that $\mu_2=\mu_3=\mu_4=\mu_5=0$.
In this table, the term $(2,-1,3,9)$ in the LHS of the second row stands for the $K$-types
$$
[a+\textbf{2},0,0,0,0,f\textbf{-1},\textbf{3}-4a-2f-6n]=[a',0,0,0,0,f',\textbf{9}-4a'-2f'-6n],
$$
where $a'=a+2$ and $f'=f-1$, with $a, f\in\bbN$. Collecting the bolded numbers gives $(2,-1,3,9)$. Thus it completely cancels out with the RHS term $(2,0,1,9)$, which stands for the $K$-types
$$
[a+\textbf{2},0,0,0,0,f+\textbf{0}, \textbf{1}-4a-2f-6n]=[a',0,0,0,0,f,\textbf{9}-4a'-2f-6n],
$$
where $a'=a+2$, with $a, f\in\bbN$. Whenever $a$ or $f$ runs over some range other than $\bbN$, we will explicitly write it out.

In Table \ref{table-E7-Wallach-0000}, almost each LHS completely cancels out with its RHS counterpart, except for the first row and the fourth row. In the first row, after cancellation with the RHS, the LHS leaves us with
\begin{equation}\label{0000-pos}
[0, 0, 0, 0, 0, f, 15-2f-6n],
\end{equation}
where $a=0, f\in\bbN$. In the fourth row, after cancellation with the LHS, the RHS leaves us with $[0, 0, 0, 0, 0, f-1, 11-2f-6n]$, where $a=0, f\in\bbN$.
Note that
$$
[0, 0, 0, 0, 0, f-1, 11-2f-6n]=[0,0,0,0,0,f', 9-2f'-6n],
$$
where $f'=f-1$ runs over $\bbN$ as well. This term further cancels out with \eqref{0000-pos}, and eventually we are left with $[0,0,0,0,0,f,15-2f]$, where $f\in\bbN$, in the LHS.

When $(\mu_2, \mu_3, \mu_4, \mu_5)$ take other values, one can proceed similarly to conclude that the positive part completely cancels out with the negative part. Here we only present the $(0,0,1,0)$ case in Table \ref{table-E7-Wallach-0010}, and the $(1,0,0,1)$ case in Table \ref{table-E7-Wallach-1001}. Other eight cases are omitted.

Finally, subtracting $-\rho_n=-3\zeta=[0,0,0,0,0,-27]$ leads us to \eqref{K-E7-Wallach-1}.
\end{proof}

\clearpage
\begin{table}
\caption{Involved $K$-types $\mu$ for the Wallach module $L(-4\zeta)$ of $E_{7(-25)}$
\\ with $(\mu_2,\mu_3,\mu_4,\mu_5)=(0,0,0,0)$}
\begin{tabular}{cc}
the $K$-types of $N(\xi)$  &  the $K$-types of $N(\eta)$  \\
\hline
$*(0,0,15,15)$ & $(1,0,11,15)$  \\
$(2,-1,3,9)$ & $(2,0,1,9)$\\
$(0,0,9,9), f\geq 1$&$(-1,1,11,9)$\\
$(0,0,9,9), a\geq 1$&$*(0,-1,11,9)$\\
$(1,1,3,9)$&$(1,0,5,9), f\geq 1$\\
$(1,-2,3,3)$&$(1,-1,1,3)$\\
$(2,-1,-3,3)$&$(2,-2,-1,3)$\\
$(0,0,3,3), f\geq 1$&$(-1,1,5,3)$\\
$(0,0,3,3), a \geq 1$&$(0,-1,5,3), a\geq 1$\\
$(0,0,3,3), a,f \geq 1$&$(0,1,1,3), a\geq 1$\\
$(1,1,-3,3)$&$(1,0,-1,3), f\geq 1$\\
$(-1,2,3,3)$&$(0,2,-1,3)$\\
$(1,-2,-3,-3)$&$(0,-1,-1,-3), a\geq 1$\\
$(-1,-1,3,-3)$&$(0,-2,1,-3)$\\
$(0,0,-3,-3), f\geq 1$&$(-1,0,1,-3), f\geq 1$\\
$(0,0,-3,-3), a\geq 1$&$(1,-1,-5,-3)$\\
$(0,0,-3,-3), a, f \geq 1$&$(0,1,-5,-3), a\geq 1$\\
$(-2,1,3,-3)$&$(-1,1,-1,-3)$\\
$(-1,2,-3,-3)$&$(-2,2,1,-3)$\\
$(0,0,-9,-9), f\geq 1$&$(-1,0,-5,-9), f\geq 1$\\
$(0,0,-9,-9), a\geq 1$&$(1,-1,-11,-9)$\\
$(-1,-1,-3,-9)$&$(-2,0,-1,-9)$\\
$(-2,1,-3,-9)$&$(0,1,-11,-9)$\\
$(0,0,-15,-15)$&$(-1,0,-11,-15)$\\
\hline
\end{tabular}
\label{table-E7-Wallach-0000}
\end{table}

\clearpage
\begin{table}
\caption{Involved $K$-types $\mu$ for the Wallach module $L(-4\zeta)$ of $E_{7(-25)}$
\\ with $(\mu_2,\mu_3, \mu_4,\mu_5)=(0,0,1,0)$ }
\begin{tabular}{cc}
the $K$-types of $N(\xi)$  &  the $K$-types of $N(\eta)$  \\
\hline
$(-1,-1,9,3)$ & $(0,-1,5,3)$\\
$(-1,-1,3,-3)$&$(-2,0,5,-3)$\\
$(-1,-1,3,-3)$&$(-1,-2,5,-3)$\\
$(-1,-1,-3,-9)$&$(0,-2,-5,-9)$\\
$(-1,-1,-3,-9)$&$(-1,0,-5,-9)$\\
$(-1,-1,-9,-15)$&$(-2,-1,-5,-15)$\\
\hline
\end{tabular}
\label{table-E7-Wallach-0010}
\end{table}

\begin{table}
\caption{Involved $K$-types $\mu$ for the Wallach module $L(-4\zeta)$ of $E_{7(-25)}$
\\ with $(\mu_2,\mu_3,\mu_4,\mu_5)=(1,0,0,1)$ }
\begin{tabular}{cc}
the $K$-types of $N(\xi)$  &  the $K$-types of $N(\eta)$  \\
\hline
$(0,-2,3,-1)$&$(-1,-1,5,-1)$ \\
$(-2,-1,3,-7)$&$(-1,-2,1,-7)$ \\
$(0,-2,-3,-7)$&$(-1,-1,-1,-7)$ \\
$(-2,-1,-3,-13)$&$(-1,-2,-5,-13)$ \\
\hline
\end{tabular}
\label{table-E7-Wallach-1001}
\end{table}

%\medskip
%\centerline{\scshape Funding}
\section*{Fundings}
Dong is supported by NSFC grant 12171344.
Luan is supported by the Shandong Provincial Natural Science Foundation under Grant ZR2022QA056.

%\medskip
%\centerline{\scshape Acknowledgements}
\section*{Acknowledgements}
Luan thanks helpful discussions with Dr. Xuanzhong Dai of Kyoto University.

\end{document}